\documentclass[12pt,a4paper,reqno]{article}

\usepackage{amssymb,bm}
\usepackage{latexsym}
\usepackage{exscale}
\usepackage{mathrsfs}
\usepackage{color}
\usepackage{amsfonts,amsbsy, bm}

\usepackage{hyperref}
\usepackage[left=2.5cm,right=2.5cm,top=2cm,bottom=2cm]{geometry}

\usepackage{amssymb}
\usepackage{amsrefs}
\usepackage{amsmath}

\newcommand {\SX} {{\mathbb X}}
\newcommand {\A} {{\mathbb A}}
\newcommand {\SH} {{\mathbb H}}

\newcommand{\one}{\mathbf{1}}
\newcommand{\xx}{\mathbf{x}}
\def\supp{\mathop{\rm supp}}

\numberwithin{equation}{section}
\newtheorem{theorem}{Theorem}[section]
\newtheorem{lemma}[theorem]{Lemma}

\newtheorem{Remark}[theorem]{Remark}

\newtheorem{definition}[theorem]{Definition}
\newtheorem{example}[theorem]{Example}

	\title{Greedy algorithms: a review and open problems}

\date{}
\author{Andrea Garc\'ia (Universidad San Pablo-CEU and CUNEF Universidad)}

\begin{document}

	\maketitle
	%\affil{\orgdiv{Departamento de Métodos Cuantitativos}, \orgname{CUNEF Universidad}, \orgaddress{\city{Madrid}, \postcode{28040},\country{Spain}}}
	
	%\affil{\orgdiv{Universidad San Pablo CEU}, \orgname{CEU Universities}, \orgaddress{ \city{Madrid}, \postcode{28003}, \country{Spain}}}
	
	
	\begin{abstract}Greedy algorithms are a fundamental category of algorithms in mathematics and computer science, characterized by their iterative, locally optimal decision-making approach, which aims to find global optima. In this review, we will discuss two greedy algorithms. First, we will talk about the so-called Relaxed Greedy Algorithm in the context of dictionaries in Hilbert spaces analyzing the optimality of the definition of this algorithm and, next, we give a general overview of the Thresholding Greedy Algorithm and the Chebyshev Thresholding Greedy Algorithm with regard to bases in $p$-Banach spaces with $0<p\leq 1$. In both cases, we pose some questions for future research.
	\end{abstract}

	
	
	

	\section{Introduction and Background}
	For over twenty years, Greedy Approximation Theory has been one of the most important areas in the field of Non-linear Approximation Theory. Within the field, there are many researchers working on and analyzing greedy-type algorithms, but one of the leading figures who has been investigating greedy algorithms for over 20 years is Vladimir Temlyakov (see \cite{T2011}). The greedy algorithms introduced by V. N. Temlyakov have a significant impact on approximation theory and on the understanding of signal and data compression. 
	
	In general, greedy algorithms are a family of approximation algorithms that make decisions at each step based on the choice that seems best at that moment, with the hope of finding a globally optimal solution. Although these algorithms do not always guarantee the optimal solution in all problems, they are especially valuable in scenarios where finding an exact optimal solution is computationally intractable. Some fields where these algorithms appear are the following:
	
	\begin{itemize}
		\item Greedy algorithms in the field of Hilbert and Banach spaces with regard to dictionaries are an important technique in data and signal compression (see \cite{BT}).
		%\item {Solving Variational Problems:} recently, versions of these algorithms have been adapted to address variational problems, providing new perspectives and solution techniques.
		\item Greedy algorithms regarding to bases in Banach spaces aim to build approximations by iteratively selecting the “most significant” coefficients respect to this basis. This process is particularly relevant in the study of sparse approximations, where one seeks to represent a function or signal with as few basis elements as possible while maintaining a certain level of accuracy (see \cite{Tsparse}).
	\end{itemize}
	
	In this review, we will analyze two different algorithms. The first one, is the so called ``Relaxed Greedy Algorithm" $(G_m^r)_{m\in\mathbb N}$ where, working with a dictionary in a Hilbert space, for $m\geq 2$,
	$$G_m^r(f):=\left(1-\dfrac{1}{m}\right)G_{m-1}^r(f)+\dfrac{1}{m}g(R_{m-1}^r(f)),$$
	and
	$$R_m^r(f)=R_m^r(f,\mathcal D):=f-G_m^r(f),$$
	where $g=g(f)\in\mathcal D$ be an element which maximized $\langle f,g\rangle$. For $m=0$, $R_0(f):=f$ and $G_0(f):=0$ and, for $m=1$, we need to take the elements given by the so called "Pure Greedy Algorithm" (see Section \ref{sec3} for the precise definition).  Using this algorithm, in \cite{DT} we can find that
	$$\Vert f-G_m^r(f)\Vert\leq \dfrac{2}{\sqrt{m}},\; m=1,2,....,$$
	for every $f\in A_1(\mathcal D)$, where $A_1(\mathcal D)$ is the class of functions $f=\sum_{n\in A}a_k g_k$ with $g_k\in\mathcal D$, $\vert A\vert<\infty$ and $\sum_{n\in A}\vert a_n\vert\leq 1$.
	Given this result, it is natural to ask whether it is possible to improve the previous inequality and go from  $m^{-1/2}$  to a better exponent. After analyzing this algorithm to answer this question, we will redefine the algorithm in the following sense:
	$$G_m^r(f):=\left(1-\dfrac{1}{m^\alpha}\right) G_{m-1}^r(f)+\dfrac{1}{m^{\alpha}}g(R_{m-1}^r(f)).$$
	Under this new extension, we prove that the inequality proved by Temlyakov for the Relaxed Greedy Algorithm is the optimal one here, that is, the optimal exponent $\alpha$ is $1$.
	
	The second algorithm that we study is the Thresholding Greedy Algorithm (TGA) $(\mathcal G_m)_{m\in\mathbb N}$ in general $p$-Banach (or quasi-Banach) spaces with regard to bases. This algorithm was introduced around 25 years ago by S. V. Konyagin and V. N. Temlyakov in \cite{KT} for Banach spaces and studied in \cite{AABW} in the context of quasi-Banach spaces: given a semi-normalized basis $\mathcal B=(\xx_n)_{n\in\mathbb N}$ in a quasi-Banach space $\SX$, if $f=\sum_{n}a_n(f)\xx_n$, the TGA selects the largest coefficients of $f$, that is,
	$$\mathcal G_m(f):=\sum_{n\in G}a_n(f)\xx_n,$$
	where $\vert G\vert=m$ and
	$$\min_{n\in G}\vert a_n(f)\vert\geq \max_{n\not\in G}\vert a_n(f)\vert.$$
	In Section \ref{sec4} we give a general review about different greedy-like bases to analyze the convergence of the TGA, but we will focus our attention and we will pose some open questions regarding the relation between almost-greedy and semi-greedy bases. Moreover, if we say that a basis is almost-greedy if the TGA produces the best approximation by projections, that is,
	$$\Vert f-\mathcal G_m(f)\Vert \approx \inf_{\vert A\vert\leq m}\left\Vert f-\sum_{n\in A}a_n(f)\xx_n\right\Vert,\; \forall m\in\mathbb N, \forall f\in\SX,$$
	we will prove that this condition is equivalent to work only with elements $f\in\A$ where $\A$ is
	$$\A:=\lbrace f\in\SX : \vert a_n(f)\vert\neq \vert a_j(f)\vert\; \forall n\neq j, n,j\in\supp(f)\rbrace.$$ 
	
	Throughout this article, the symbol $a_j\lesssim b_j$ means that there is a positive constant $C$ such that $a_j\leq C\, b_j$ for $j\in J$ and we say that $(a_j)_{j\in J}$ and $(b_j)_{j\in J}$ are equivalent, and we write $a_j \approx b_j$, when $a_j \lesssim b_j$ and $b_j\lesssim a_j$. If $A$ is a finite subset of $\mathbb N$, $\vert A\vert$ is the cardinality of $A$. Also, we define the following set
	$$\mathcal E_A:=\left\lbrace \boldsymbol{\varepsilon}=(\varepsilon_n)_{n\in A} : \vert\delta_n\vert=1\; \forall n\in A\right\rbrace,$$
	and
	$$\one_{\boldsymbol{\varepsilon} A}[\mathcal B,\SX]=\one_{\boldsymbol{\varepsilon} A}:=\sum_{n\in A}\varepsilon_n\xx_n,$$
	where $\one_{\boldsymbol{\varepsilon} A}$ are the so-called indicator sums. When $\boldsymbol{\varepsilon}\equiv \boldsymbol{1}=(1)_{n\in A}$, we use the notation $\one_A$.

	\section{Bases and dictionaries in quasi-Banach spaces}
	
	To introduce and analyze the different greedy algorithms, we need to fix some concepts. One the one hand, a \textbf{quasi-Banach space} is a complete quasi-normed vector space over $\mathbb F=\mathbb R$ or $\mathbb C$, where a quasi-norm on a vector space \(\SX\) is a function \(\|\cdot\|: \SX \to \mathbb{R}^+\) that satisfies the following properties:
	
	\begin{enumerate}
		\item[Q1)] \(\|x\| = 0 \iff x = 0\).
		\item[Q2)] \(\|\alpha x\| = |\alpha| \|x\|\) for all $\alpha\in\mathbb F$ and $x\in\SX$.
		\item[Q3)] There exists a constant \(C \geq 1\) such that \(\|x + y\| \leq C (\|x\| + \|y\|)\) for all $x,y\in\SX$.
	\end{enumerate}
	
	If \(C = 1\), the quasi-norm \(\|\cdot\|\) is a norm, and the quasi-Banach space is called a \textbf{Banach space}. A typical example of a quasi-Banach space is \(\ell_p\),  \(0 < p < \infty\), where this space contains the sequences \(x = (x_n)\) such that \(\sum_{n=1}^{\infty} |x_n|^p < \infty\) and the quasi-norm is \(\|x\|_p = \left( \sum_{n=1}^{\infty} |x_n|^p \right)^{1/p}\). For $p\geq 1$, $\ell_p$ is a Banach space.

	On the other hand, a \textbf{\(p\)-Banach space}  is a generalization of a Banach space where the norm satisfies a modified version of the triangle inequality. For \(0 < p \leq 1\), a \(p\)-norm on a vector space \(\SX\) is a function \(\|\cdot\|: \SX \to \mathbb{R}^+\) that satisfies the following properties for all \(x, y \in \SX\) and all scalars \(\alpha \in \mathbb{R}\) or $\mathbb C$:
	
	\begin{enumerate}
		\item[P1)] \(\|x\| = 0 \iff x = 0\).
		\item[P2)] \(\|\alpha x\| = |\alpha| \|x\|\).
		\item[P3)] \(\|x + y\|^p \leq \|x\|^p + \|y\|^p\).
	\end{enumerate}
	
	If \(p = 1\), the \(p\)-norm is a norm, and the \(p\)-Banach space is a Banach space. Also, since P3) implies Q3), a $p$-Banach space is a quasi-Banach space. Thanks to the Aoki-Rolewicz Theorem (see \cite{Aoki,Rolewicz}), any quasi-Banach space $\SX$ is $p$-convex, that is,
	$$\left\Vert \sum_{j=1}^n x_j\right\Vert\leq C\left(\sum_{j=1}^n \Vert x_j\Vert^p\right)^{1/p}, n\in\mathbb N, x_j\in \SX.$$
	This way, $\SX$ becomes $p$-Banach under a suitable renorming, for some $0<p\leq 1$. 
	
	Despite the fact that a quasi-norm is not necessarily a continuous map, the reverse triangle law
	\begin{eqnarray}
		\vert\;\Vert f\Vert^p -\Vert g\Vert^p\;\vert\leq \Vert f-g\Vert^p,\; f,g\in \SX,
	\end{eqnarray} 
	implies that a $p$-norm is continuous. Hence, the Aoki-Rolewicz Theorem implies that any quasi-Banach space can be equipped with a continuous quasi-norm.

	A particular case of interest of Banach spaces is a Hilbert space. A \textbf{Hilbert space} is a complete inner product space, that is, a vector space \(\mathbb H\) equipped with an inner product \(\langle \cdot, \cdot \rangle: \SH \times \SH \to \mathbb{R}\) or \(\mathbb{C}\) that satisfies the following properties for all \(x, y, z \in \SH\) and all scalars \(\alpha \in \mathbb{R}\) or \(\mathbb{C}\):
	
	\begin{enumerate}
		\item \(\langle x, x \rangle \geq 0\) and \(\langle x, x \rangle = 0 \iff x = 0\).
		\item \(\langle x, y \rangle = \overline{\langle y, x \rangle}\).
		\item \(\langle \alpha x, y \rangle = \alpha \langle x, y \rangle\).
		\item \(\langle x + y, z \rangle = \langle x, z \rangle + \langle y, z \rangle\).
	\end{enumerate}
	
	The norm induced by the inner product is defined as \(\|x\| = \sqrt{\langle x, x \rangle}\). A Hilbert space is an inner product space that is complete with respect to this norm. Some typical examples are the following ones:
	
	\begin{itemize}
		\item \text{\(\ell_2\) space:} the space of square-summable sequences \(x = (x_n)\) such that \(\sum_{n=1}^{\infty} |x_n|^2 < \infty\) is a Hilbert space with the inner product \(\langle x, y \rangle = \sum_{n=1}^{\infty} x_n \overline{y_n}\).
		\item \text{\(L^2\) space:} the space of square-integrable functions \(f: [a,b] \to \mathbb{R}\) (or \(\mathbb{C}\)) such that \(\int_a^b |f(x)|^2 \, dx < \infty\) is a Hilbert space with the inner product \(\langle f, g \rangle = \int_a^b f(x) \overline{g(x)} \, dx\).
	\end{itemize}
	
	In the field of Functional Analysis, Banach spaces play a fundamental role due to their complete and normed structure, making them a natural framework for studying a wide range of mathematical problems. A key concept within these spaces is the notion of a basis, which provides a powerful tool for representing and analyzing the elements of the space. Therefore, we will introduce two concepts of a basis that will be of great utility.

	A \textbf{Schauder basis} for a quasi-Banach space \(\SX\) is a sequence \((\xx_n)_{n \in \mathbb{N}}\) in \(\SX\) such that for every \(x \in \SX\), there exists a unique sequence of scalars \((a_n)_{n \in \mathbb{N}}\) such that
	\[
	x = \sum_{n=1}^{\infty} a_n \xx_n,
	\]
	where the series converges in the norm of \(\SX\). In other words, for every \(x \in \SX\), there exists a unique sequence of coefficients $(a_n)_{n\in\mathbb N}$ in $\mathbb R$ or \(\mathbb{C}\) such that \(x\) can be represented as a convergent series of the basis elements. Some of the most important properties of these types of bases are the following:

	\begin{itemize}
		\item \text{Uniqueness:} the representation of any element \(x \in \SX\) as a series \(\sum_{n=1}^{\infty} a_n \xx_n\) is unique.
		\item \text{Finite Dimensional Approximation:} for each \(x \in \SX\) and each \(m \in \mathbb{N}\), the partial sum \(S_m(x) = \sum_{n=1}^m a_n \xx_n\) approximates \(x\) in the norm of \(\SX\). That is, \(S_m(x) \to x\) as \(m \to \infty\).
		\item \text{Continuous Coefficient Functionals:} the coefficient functionals (or biorthogonal functionals) \(\xx_n^*\)  defined by \(\xx_n^*(x) = a_n\) (where \(a_n\) are the coefficients in the representation of \(x\)) are continuous linear functionals on \(\SX\). Under this contidion, each $x\in\SX$ could be expressed like
		$$x=\sum_{n=1}^{\infty}\mathbf x_n^*(x)\mathbf x_n.$$
	\end{itemize}
	
	A weaker condition than Schauder is the notion of a Markushevich basis. Consider a system \((\xx_n, \xx_n^*)_{n \in \mathbb{N}}\subset \SX\times \SX^*\) veriying the following conditions:
	\begin{enumerate}
		\item $\xx_n^*(\xx_m)=\delta_{n,m}$ for all \(n, m \in \mathbb{N}\), where \(\delta_{nm}\) is the Kronecker delta.
		\item The linear span of \((\xx_n)_{n \in \mathbb{N}}\) is dense in \(\SX\).
		\item The linear span of \((\xx_n^*)_{n \in \mathbb{N}}\) is weak*-dense in \(\SX^*\).
	\end{enumerate}
	Under these conditions, $\mathcal B=(\xx_n)_{n\in\mathbb N}$ is called a  \textbf{Markushevich basis} with the dual basis $\mathcal B^*=(\xx_n^*)_{n\in\mathbb N}$ (we remark here that the dual basis is also unique). Of course, every Schauder basis is a Markushevich basis, but not every Markushevich basis is Schauder. The difference lies in the convergence of the series representation: Markushevich bases do not require norm convergence, while Schauder bases do.  An example of that could be the trigonometric system in $\mathcal{C}([0,1])$.
	
	In this paper, we will assume that if $(\mathbf x_n)_{n\in\mathbb N}$ is a Markushevich or a Schauder basis, it is also semi-normalized, that is, there are two positive constants $\mathbf c_1,\mathbf c_2>0$ such that
	$$0<\mathbf c_1\leq \inf_n\min\lbrace\Vert \xx_n\Vert,\Vert \xx_n^*\Vert\rbrace\leq \sup_n\max\lbrace\Vert \xx_n\Vert,\Vert \xx_n^*\Vert\rbrace\leq\mathbf c_2<\infty.$$
	Also, as notation, we denote the support of $x\in\mathbb X$ as the set
	$$\text{supp}(x):=\lbrace n\in\mathbb N : \mathbf x_n^*(x)\neq 0\rbrace.$$
	
	Finally, we will talk about dictionaries. A \textbf{dictionary} \(\mathcal{D}\) is a subset of a quasi-Banach spaces \(\SX\) with the following properties:
	\begin{enumerate}
		\item Each atom \(g \in \mathcal{D}\) has norm 1, i.e., \(\|g\| = 1\) for all \(g \in \mathcal{D}\) and if $g\in\mathcal D$ then $-g\in\mathcal D$.
		\item The closed linear span of \(\mathcal{D}\) is dense in \(\SX\). In other words, for every \(x \in \SX\) and every \(\epsilon > 0\), there exist \(g_1, g_2, \ldots, g_n \in \mathcal{D}\) and scalars \(a_1, a_2, \ldots, a_n\) such that
		\[
		\left\| x - \sum_{i=1}^n a_i g_i \right\| < \epsilon.
		\]
	\end{enumerate}
	In the case for an orthonormal basis $\lbrace h_k\rbrace_{k=1}^\infty$ in a Hilbert space, a dictionary could be the collection $\mathcal D=\lbrace \pm h_k\rbrace_{k=1}^\infty$. Other examples to construct dictionaries are:
	\begin{itemize}
		\item \textbf{Haar Dictionary:} in the space \(L^2([0,1])\), the Haar wavelet system. The atoms are the Haar wavelets, which are piecewise constant functions that form an orthonormal basis for \(L^2([0,1])\).
		\item \textbf{Overcomplete Dictionary:} in the context of signal processing, an overcomplete dictionary might consist of a union of several bases (e.g., wavelet bases) to provide more flexibility in representing signals.
	\end{itemize}
	
	Related to the properties of dictionaries, we can mention the following:
	
	\begin{itemize}
		\item \textbf{Redundancy:} dictionaries are often redundant, meaning that there can be multiple representations of the same element in terms of the dictionary atoms. This redundancy can be advantageous in applications such as \textit{sparse} approximation.
		\item \textbf{Flexibility:} dictionaries are not required to be bases, allowing more flexibility in the choice of atoms. This makes them suitable for various approximation and signal processing tasks where traditional bases may not be ideal.
		\item \textbf{Approximation:} dictionaries are used in various greedy algorithms for approximation, where elements of \(X\) are approximated by selecting atoms from \(\mathcal{D}\) based on certain criteria (e.g., matching pursuit and orthogonal matching pursuit).
	\end{itemize}

	\section{Greedy algorithms for dictionaries in Hilbert spaces}\label{sec3}
	Let $\mathbb H$ be a Hilbert space (over $\mathbb R$) and $\mathcal D$ a dictionary in $\mathbb H$. Related to this dictionary, we define $A_\tau(\mathcal D)$, for $\tau>0$, as the clausure in $\mathbb H$ of the following set of functions:
	$$\left\lbrace f\in\mathbb H : f=\sum_{k\in\Lambda}a_k g_k, \; g_k\in\mathcal D,\; \vert\Lambda\vert<\infty\;\text{and}\; \sum_{k\in\Lambda}\vert a_k\vert^\tau\leq 1\right\rbrace.$$
	
	To introduce the new version of the Relaxed Greedy Algorithm studied in \cite{DT}, we need to study first the Pure Greedy Algorithm.
	
	Given $f\in\mathbb H$, we let $g=g(f)\in\mathcal D$ be an element from $\mathcal D$ which maximized $\langle f,g\rangle$ (we assume for simplicity that such a maximizer exists) and
	$$G(f)=G(f,\mathcal D):=\langle f,g\rangle g,$$
	and
	$$R(f)=R(f,\mathcal D):=f-G(f).$$
	
	\textbf{Pure Greedy Algorithm:} define $R_0(f)=R_0(f,\mathcal D):=f$ and $G_0(f)=G_0(f,\mathcal D):=0$. Then, for each $m\geq 1$, 
	$$G_m(f)=G_m(f,\mathcal D):=G_{m-1}(f)+G(R_{m-1}(f)),$$
	$$R_m(f)=R_m(f,\mathcal D):=f-G_m(f)=R(R_{m-1}(f)).$$
	
	\textbf{Relaxed Greedy Algorithm:} define $R_0^r(f)=R_0^r(f,\mathcal D):=f$ and $G_0^r(f)=G_0^r(f,\mathcal D):=0$. For $m=1$, we set $G_1^r(f)=G_1^r(f,\mathcal D):=G_1(f)$ and $R_1^r(f)=R_1^r(f,\mathcal D):=R_1(f)$. Then, for $m\geq 2$,
	$$G_m^r(f)=G_m^r(f,\mathcal D):=\left(1-\dfrac{1}{m}\right)G_{m-1}^r(f)+\dfrac{1}{m}g(R_{m-1}^r(f)),$$
	$$R_m^r(f)=R_m^r(f,\mathcal D):=f-G_m^r(f).$$
	
	Using this algorithm, in \cite{DT} we can find the following result.
	
	\begin{theorem}\label{dtt}
		Let $\mathcal D$ be a dictionary in a Hilbert space. Then,
		$$\Vert f-G_m^r(f)\Vert\leq \dfrac{2}{\sqrt{m}},\; m=1,2,....,$$
		for every $f\in A_1(\mathcal D)$.
	\end{theorem}

	Is it possible to improve the last bound with a better power? With the aim of investigating whether the previous result can be improved, we will introduce a “power” factor in the definition of the algorithm. For that, consider a real number $\alpha\geq 0$. 
	
	\textbf{Power-Relaxed Greedy Algorithm:} define $\mathcal R_0^{r}(f)=\mathcal R_0^{r}(f,\mathcal D):=f$ and $\mathcal T_0^{r}(f)=\mathcal T_0^{r}(f,\mathcal D):=0$. For $m=1$, we set $\mathcal T_1^{r}(f)=\mathcal T_1^{r}(f,\mathcal D):=G_1(f)$ and $\mathcal R_1^{r}(f)=\mathcal R_1^{r}(f,\mathcal D):=R_1(f)$. Then, for $m\geq 2$,
	$$\mathcal T_m^{r}(f)=\mathcal T_m^{r}(f,\mathcal D):=\left(1-\dfrac{1}{m^\alpha}\right)\mathcal T_{m-1}^{r}(f)+\dfrac{1}{m^\alpha}g(\mathcal R_{m-1}^{r}(f)),$$
	$$\mathcal R_m^{r}(f)=\mathcal R_m^{r}(f,\mathcal D):=f-\mathcal T_m^{r}(f).$$
	
	We show the following theorem.
	
	\begin{theorem}\label{than}
		Let $\mathcal D$ be a dictionary in a Hilbert space. For $\alpha\leq 1$ and $m=1,2,...$,
		$$\Vert f-\mathcal T_m^r(f)\Vert^2\leq \dfrac{4}{m^{\alpha}},$$
		for every $f\in A_1(\mathcal D)$.
	\end{theorem}
	
	In \cite{DT}, the authors proved that, for a constant $A>0$ and any sequence of positive real scalars $(a_n)_n\in\mathbb N$ satisfying the inqueality
	$$a_m\leq \left(1-\dfrac{2}{m}\right)a_{m-1}+\dfrac{A}{m^{2}},\; m=2,3,....$$
	with $a_1\leq A$, it is possible to obtain that
	$$a_m\leq \dfrac{A}{m},\, m=1,2,...,$$
	and this is the main key to show Theorem \ref{dtt}. Here, we extend this inequality.
	
	\begin{lemma}\label{technical}
		Let $(a_n)_{n\in\mathbb N}$ be a sequence of positive real numbers such that $a_1\leq A$ for some $A>0$ and
		$$a_m\leq \left(1-\dfrac{2}{m^\alpha}\right)a_{m-1}+\dfrac{A}{m^{2\alpha}},\; m=2,3,....$$
		Hence, if $\alpha\leq 1$, $a_m\leq \dfrac{A}{m^\alpha}$.
	\end{lemma}
	\begin{Proof}
		Of course, for $m=1$ is true by hypothesis so we proceed by induction. Assume that
		$$a_{m-1}\leq  \dfrac{A}{(m-1)^\alpha}.$$
		Then,
		\begin{eqnarray*}
			a_m&\leq& \left(1-\dfrac{2}{m^\alpha}\right)a_{m-1}+\dfrac{A}{m^{2\alpha}}\\
			&\leq& \left(1-\dfrac{2}{m^\alpha}\right)\dfrac{A}{(m-1)^\alpha}+\dfrac{A}{m^{2\alpha}}\\
			&=&A\dfrac{m^{2\alpha}-2m^\alpha+(m-1)^\alpha}{m^{2\alpha}(m-1)^\alpha}.
		\end{eqnarray*}
		Hence,
		\begin{eqnarray}\label{imp}
			\nonumber	\dfrac{m^{2\alpha}-2m^\alpha+(m-1)^\alpha}{m^{2\alpha}(m-1)^\alpha}\leq \dfrac{1}{m^{\alpha}}&\Leftrightarrow& m^{2\alpha}-2m^\alpha+(m-1)^\alpha\leq m^\alpha(m-1)^\alpha\\
			\nonumber	&\Leftrightarrow& m^\alpha -2 +\left(\dfrac{m-1}{m}\right)^\alpha \leq (m-1)^\alpha\\
			&\Leftrightarrow& m^\alpha-(m-1)^\alpha\leq 2-\left(\dfrac{m-1}{m}\right)^\alpha.
		\end{eqnarray}
		Of course, if $\alpha=1$, the inequality is true:
		\begin{eqnarray*}
			1\leq 2-\dfrac{m-1}{m}\Leftrightarrow \dfrac{m-1}{m}\leq 1\Leftrightarrow m-1\leq m,
		\end{eqnarray*}
		and this inequality is trivial. Now, if $\alpha<1$, we define $f(x)=x^\alpha -(x-1)^\alpha$. Then, it is trivial to see that this function is decreasing since 
		$$f'(x)=\alpha (x^{\alpha-1}-(x-1)^{\alpha-1})<0\Leftrightarrow \dfrac{1}{x^{1-\alpha}}\leq \dfrac{1}{(x-1)^{1-\alpha}},$$
		and since we work with $x\geq 1$, $f(x)$ is decreasing. On the other hand, $f(1)=1$, so
		$$m^\alpha-(m-1)^{\alpha}\leq 1.$$
		On the other hand,
		$$\left(\dfrac{m-1}{m}\right)^\alpha=\left(1-\dfrac{1}{m}\right)^\alpha\leq 1,$$
		then,
		$$m^\alpha-(m-1)^\alpha+\left(\dfrac{m-1}{m}\right)^\alpha\leq 2,$$
		concluding that \eqref{imp} is true. Hence, 
		$$a_m\leq \dfrac{A}{m^\alpha},\;\forall m\in\mathbb N,$$
		when $\alpha\leq 1$.
	\end{Proof}
	
	\begin{Proof}[{Proof of Theorem \ref{than}}]
		To prove this theorem, we follow the idea of Theorem \ref{dtt}. 
		\begin{eqnarray}
			\Vert f-\mathcal T_m^r(f)\Vert^2&=&\left\Vert (f-\mathcal T_{m-1}^r(f))+\dfrac{1}{m^{\alpha}}(\mathcal T_{m-1}^r(f)-g(\mathcal R_{m-1}^r(f)))\right\Vert^2\\
			&=& \Vert (f-\mathcal T_{m-1}^r(f))\Vert^2 + \dfrac{1}{m^{2\alpha}}\Vert \mathcal T_{m-1}^r(f)-g(\mathcal R_{m-1}^r(f))\Vert^2\\
			&+&2\langle f-\mathcal T_{m-1}^r(f), \mathcal T_{m-1}^r(f)-g(\mathcal R_{m-1}^r(f))\rangle.
		\end{eqnarray}
		As we can find in \cite{DT}, it is not hard to show that
		$$\langle f-\mathcal T_{m-1}^r(f), \mathcal T_{m-1}^r(f)-g(\mathcal R_{m-1}^r(f))\rangle\leq -\Vert f-\mathcal T_{m-1}^r(f)\Vert^2,$$
		so we obtain
		\begin{eqnarray}
			\Vert f-\mathcal T_m^r(f)\Vert^2\leq \left(1-\dfrac{2}{m^\alpha}\right)\Vert f-\mathcal T_{m-1}^r(f)\Vert^2+\dfrac{1}{m^{2\alpha}}\Vert \mathcal T_{m-1}^r(f)-g(\mathcal R_{m-1}^r(f))\Vert^2.
		\end{eqnarray}
		Now, we analyze the quantity $\Vert \mathcal T_{m-1}^r(f)-g(\mathcal R_{m-1}^r(f))\Vert^2$. On the one hand, $f\in\mathcal A_1(\mathcal D)$, $f=\sum_{j\in\Lambda} a_j g_j$ with $g_j\in\mathcal D$ with $\sum_{j\in\Lambda}\vert c_j\vert\leq 1$. Then,
		$$\Vert f\Vert\leq \sum_{j\in\Lambda}\vert c_j\vert\Vert g_j\Vert=\sum_{j\in\Lambda}\vert c_j\vert\leq 1.$$
		On the other hand, since $\mathcal T_1^r(f)=G_1(f)$, it is obvious that $\Vert \mathcal T_1^r(f)\Vert\leq 1$. Now, by induction, assuming that $\Vert \mathcal T_{m-1}^r(f)\Vert\leq 1$ for some $m\in\mathbb N$, then
		\begin{eqnarray*}
			\Vert \mathcal T_m^r(f)\Vert&=&\left\Vert \left(1-\dfrac{1}{(m-1)^{\alpha}}\right)\mathcal T_{m-1}^r(f)+\dfrac{1}{(m-1)^{\alpha}}g(\mathcal R_{m-1}^r(f))\right\Vert\\
			&\leq& \left(1-\dfrac{1}{(m-1)^{\alpha}}\right)\Vert \mathcal T_{m-1}^r(f)\Vert + \dfrac{1}{(m-1)^{\alpha}}\Vert g(\mathcal R_{m-1}^r(f))\Vert\\
			&\leq& 1-\dfrac{1}{(m-1)^{\alpha}}+\dfrac{1}{(m-1)^{\alpha}}=1,
		\end{eqnarray*}
		where we have used that $\Vert g(\mathcal R_{m-1}^r(f))\Vert \leq 1$ since it is an element from $\mathcal D$. Hence, $\Vert \mathcal T_m^r(f)\Vert\leq 1$ for every $m\in\mathbb N$ and $\Vert f\Vert\leq 1$. Thus,
		$$\Vert \mathcal T_{m-1}^r(f)-g(\mathcal R_{m-1}^r(f))\Vert\leq 2,$$
		obtaining
		$$\Vert f-\mathcal T_m^r(f)\Vert^2\leq \left(1-\dfrac{2}{m^\alpha}\right)\Vert f-\mathcal T_{m-1}^r(f)\Vert^2+\dfrac{4}{m^{2\alpha}}.$$
		Then, based on Lemma \ref{technical}, taking $a_m=\Vert f-\mathcal G_m^r(f)\Vert^2$ and $A=4$, we obtain, for $\alpha \geq 1$,
		$$\Vert f-\mathcal T_m^r(f)\Vert^2\leq \dfrac{4}{m^{\alpha}},$$
		but the optimal power is $\alpha=1$, where then, we are recovering the result proved in Theorem \ref{dtt}.
	\end{Proof}
	
	\textbf{Question 1.} Is it possible to find $\alpha_0>1$ such that
	$$\Vert f-\mathcal T_m^r(f)\Vert \lesssim \dfrac{1}{m^{\alpha_0/2}},\; m=1,2,..$$
	for every $f\in A_1(\mathcal D)?$
	
	\textbf{Question 2.} Is it possible to find another function $h(m)$ such that if we define, for $m\geq 2$,
	$$\mathcal T_m^r=\left(1-\dfrac{1}{h(m)}\right)\mathcal T_{m-1}^r(f)+\dfrac{1}{h(m)}g(\mathcal R_{m-1}^r(f)),$$
	such that we obtain a better bound than Theorem \ref{dtt}?

	\section{Thresholding Greedy Algorithms with respect to bases in quasi-Banach spaces}\label{sec4}
	The Thresholding Greedy Algorithm $(\mathcal G_m)_{m=1}^\infty$ (TGA for short) was first introduced by S. V. Konyagin and V. N. Temlyakov in \cite{KT}. More precisely, for a  Markushevich basis $\mathcal B=(\mathbf x_n)_{n=1}^\infty$ in a quasi-Banach space $\mathbb X$, for $x\in\mathbb X$ and $m\in\mathbb N$, 
	$$\mathcal G_m[\mathcal B,\mathbb X](x):=\mathcal G_m(x)=\sum_{n\in A}\mathbf x_n^*(x)\mathbf x_n,$$ where $A$ is any set so that $\vert A\vert = m$ and $$\min_{n\in A}\vert \mathbf x_n^*(x)\vert\geq \max_{n\not\in A}\vert \mathbf x_n^*(x)\vert.$$
	In this case, $\mathcal G_m(x)$ and the set $A=\supp(\mathcal G_m(x))$ are called a \textit{greedy sum} and a \textit{greedy set} of $f$ of order $m$, respectively.
	
	The most natural way to describe a greedy sum of an element $x\in\mathbb X$ is to take an injective map $\pi: \mathbb N\rightarrow\mathbb N$ such that $\supp(x)\subseteq\pi(\mathbb N)$ and $(\vert \mathbf x_{\pi(j)}^*(x)\vert)_{j=1}^\infty$ is non-increasing and then, consider the partial sums
	$$\mathcal G_m(x)=\sum_{j=1}^{m}\mathbf x_{\pi(j)}^*(x)\mathbf x_{\pi(j)}.$$
	Every such map $\pi$ is call a \textit{greedy ordering}.  In particular, a greedy sum is a projection $P_A(x)$, where, for $x\in\SX$ and $A$ a finite set of indices,
	$$P_A[\mathcal B,\SX](x)=P_A(x):=\sum_{n\in A}\xx_n^*(x)\xx_n.$$
	Hence, if we take $G=\lbrace \pi(1),\pi(2),\dots,\pi(m)\rbrace$, $P_G(x)$ is a greedy sum of $x$ of cardinality $m$ with $G$ the corresponding greedy set.
	
	As we can find in \cite{Wo3}, it is important to remark that $(\mathcal G_m)_{m\in\mathbb N}$ are neither linear nor continuous. For instance, we define the elements
	$$f_n := \dfrac{n^2+1}{n^2}\one_{A}+\one_{B},$$
	and
	$$g_n:=\one_{A}+\dfrac{n^2+1}{n^2}\one_{B},$$
	where $A$ and $B$ are disjoint and with cardinality $m$. Then, both elements converge to $\one_{A\cup B}$ but,
	$$\mathcal G_m(f_n)=\dfrac{n^2+1}{n^2}\one_{A},\;\;  \mathcal G_m(g_n)=\dfrac{n^2+1}{n^2}\one_{B},$$
	hence, $\mathcal G_m(f_n)\rightarrow \one_{A}$ and $ \mathcal G_m(g_n)\rightarrow \one_{B}$. Thus, $\mathcal G_m$ is not continuous. To see the non linearity, we  take, for instance, the canonical basis $\mathcal B=(\mathbf x_n)_{n\in\mathbb N}$ in the space $\ell_2$, where 
	$$\mathbf x_n=(0,0,\dots,0,\underbrace{1}_\text{position $n$},0,\dots).$$
	Define the following elements:
	$$f=\sum_{i=1}^k \mathbf x_i + \sum_{j=k+1}^\infty \dfrac{1}{j^3}\mathbf x_j,\;\;g=-\sum_{i=1}^k \mathbf x_i + \sum_{j=k+1}^\infty \dfrac{1}{j^3}\mathbf x_j.$$
	Then, $\mathcal G_k(f)=\sum_{i=1}^k \mathbf x_i$ and $ \mathcal G_k(g)=-\sum_{i=1}^k\mathbf x_i$, but $ \mathcal G_k(f+g)=2\sum_{j=k+1}^{2k}\dfrac{1}{j^3}\mathbf x_j$, that is different from $\mathcal G_k(f)+\mathcal G_k(g)=\mathbf 0$.
	
	Related to the convergence of the TGA, we can find different type of bases. We start by quasi-greedy bases introduced the first time in \cite{KT} in the context of Banach spaces and studied before in quasi-Banach spaces in \cite{AABW,W2000}.
	
	\begin{definition}
		We say that a Markushevich basis $\mathcal B$ in a quasi-Banach space $\mathbb X$ is quasi-greedy if there is a positive constant $C$ such that for every $x\in\SX$,
		$$\Vert x-P_A(x)\Vert\leq C\Vert x\Vert,$$
		whenever $A$ is a finite greedy set of $x$.
	\end{definition}
	
	As we can study in \cite{AABW, W2000}, this notion is the weaker condition related to the convergence of the algorithm in the sense that it is possible to prove that the basis is quasi-greedy if and only if, for every $x\in\SX$, the series $\sum_{n=1}^{\infty}\xx_{\pi(n)}^*(x)\xx_{\pi(n)}$ converges to $x$. 
	
	An stronger condition than quasi-greediness is unconditionality.
	
	\begin{definition}
		We say that a Markushevich basis $\mathcal B$ in a quasi-Banach space $\mathbb X$ is unconditional if there is a positive constant $C$ such that
		\begin{eqnarray}\label{defuncond}
			\Vert x-P_B(x)\Vert\leq C\Vert x\Vert,\; \forall \vert B\vert<\infty, \forall x\in\mathbb X.
		\end{eqnarray}
		The least constant verifying \eqref{defuncond} is denoted by $K[\mathcal B,\mathbb X]=K$ and we say that $\mathcal B$ is $K$-unconditional.
	\end{definition}
	
	It is clear then that any unconditional basis is quasi-greedy, but the converse is false in general as we can see in the following example (see \cite{KT}).
	
	\begin{example}\label{ex1}
		Consider the space $\SX$ of all sequences $\mathbf a=(a_n)_{n\in\mathbb N}\in c_0$ with norm
		$$\Vert\mathbf a\Vert=\left\lbrace \left(\sum_{j=1}^\infty a_j^2\right)^{1/2},\, \sup_{m\geq 1}\left\vert \sum_{j=1}^m \dfrac{a_j}{\sqrt{j}}\right\vert\right\rbrace.$$
		Take now the canonical basis $(\xx_n)_{n\in\mathbb N}$ in this space. The proof that the canonical basis is quasi-greedy could be found in \cite{KT} or \cite{BBGHO}. To show that the basis is not unconditional, we can take the element 
		$$\mathbf f:=\sum_{j=1}^{2m}\dfrac{(-1)^{j}}{\sqrt{j}}\xx_n.$$
		On the hand, it is clear that
		$$\Vert \mathbf f\Vert= \left(\sum_{j=1}^{2m} \left(\dfrac{1}{\sqrt{j}}\right)^2\right)^{1/2}\approx \sqrt{\ln(m+1)}.$$
		On the other hand, taking the set
		$$B=\lbrace 1,\dots, 2m\rbrace\cap 2\mathbb Z,$$
		$$\Vert P_B(\mathbf f)\Vert\gtrsim \sum_{j=1}^m\dfrac{1}{j}\approx \ln(m+1).$$
		Hence,
		$$\dfrac{\Vert P_B(\mathbf f)\Vert}{\Vert\mathbf f\Vert}\gtrsim \sqrt{\ln(m+1)},$$
		so the basis is not unconditional.
	\end{example}
	
	Now, since we have seen that quasi-greediness is the weakest condition respect to the convergence of the TGA, we introduce a condition that is the strongest one: greediness (\cite{KT}).
	
	\begin{definition}
		We say that a Markushevich basis $\mathcal B$ in a quasi-Banach space $\mathbb X$ is greedy if there is a positive constant $C$ such that for every $x\in\SX$,
		\begin{eqnarray}\label{defgreedy}
			\Vert x-P_A(x)\Vert\leq C\inf\left\lbrace\left\Vert x-\sum_{n\in B}b_n\xx_n\right\Vert, b_n\in\mathbb F, \vert B\vert\leq \vert A\vert\right\rbrace,
		\end{eqnarray}
		whenever $A$ is a finite greedy set of $x$. The least constant verifying \eqref{defgreedy} is denoted by $C_g[\mathcal B,\mathbb X]=C_g$ and we say that $\mathcal B$ is $C_g$-greedy.
	\end{definition}
	
	In other words, we say that a basis is greedy when the TGA produces the best approximation, meaning that in the context of the best approximation error $\sigma_m(x)$ where
	$$\sigma_m[\mathcal B,\SX](x)=\sigma_m(x):=\inf\left\lbrace\left\Vert x-\sum_{n\in B}b_n\xx_n\right\Vert, b_n\in\mathbb F, \vert B\vert\leq m\right\rbrace,$$
	the best sum we can write would be the greedy sum (up to an absolute constant).  As we have commented, this notion was introduced the first time in \cite{KT} in the context of Schauder bases in Banach spaces, but in \cite{AABW}, the authors analyzed these bases for general Markushevich bases in quasi-Banach spaces.
	
	Of course, every greedy basis is quasi-greedy since $\sigma_m(x)\leq\Vert x\Vert$, but the converse is false in general. To study one example of quasi-greedy that is not a greedy basis, we need the main characterization of greediness introduced in \cite{KT}  based on unconditionality and democracy.
	
	\begin{definition}
		We say that a Markushevich basis $\mathcal B$ in a quasi-Banach space $\mathbb X$ is super-democratic if there is $C>0$ such that
		\begin{eqnarray}\label{seqdem}
			\Vert\one_{\boldsymbol{\varepsilon} A}\Vert\leq C\Vert \one_{\boldsymbol{\eta} B}\Vert,
		\end{eqnarray}
		for every par of finite sets $\vert A\vert\leq\vert B\vert$ and any choice of signs $\boldsymbol{\varepsilon}\in\mathcal E_A$ and $\boldsymbol{\eta}\in\mathcal E_B$.
		The least constant verifying \eqref{seqdem} is denoted by $\Delta_s[\mathcal B,\mathbb X]=\Delta_s$ and we say that $\mathcal B$ is $\Delta_s$-super-democratic. If \eqref{seqdem} is satisfied for $\boldsymbol{\varepsilon}\equiv \boldsymbol{\eta}\equiv \boldsymbol{1}$, we say that the basis is $\Delta_d$-democratic.
	\end{definition}
	
	\begin{theorem}[{\cite{KT,AABW}}]
		Let $\mathcal B$ be a basis in a quasi-Banach space. The following are equivalent:
		\begin{itemize}
			\item $\mathcal B$ is greedy.
			\item $\mathcal B$ is unconditional and democratic.
			\item $\mathcal B$ is unconditional and super-democratic.
		\end{itemize}
	\end{theorem}
	
	\begin{example}
		There are several examples of greedy bases. For instance, if we have an orthonormal basis in a Hilbert space $\SH$, for every $x\in\mathbb H$,
		$$\Vert x-P_G(x)\Vert=\sigma_{\vert G\vert}(x),$$
		whenever $G$ is a finite greedy set of $x$.
		
		Another example could be the canonical basis $\mathcal B=(\xx_n)_{n\in\mathbb N}$ in the space of sequences $\ell_p$ with $1\leq p<\infty$. The proof is so easy: on the one hand,
		$$\Vert\one_A\Vert^p=\vert A\vert^p,\; \forall \vert A\vert<\infty,$$
		so the basis is $1$-democratic. On the other hand, for $x\in\SX$ and $\vert A\vert<\infty$,
		$$\Vert x-P_A(x)\Vert^p= \sum_{j \not\in A} \vert \xx_j^*(x)\vert^p\leq  \sum_{j=1}^\infty \vert\xx_j^*(x)\vert^p=\Vert x\Vert^p,$$
		so the basis is $1$-unconditional.
	\end{example}
	
	Now, as an intermediate notion between greediness and quasi-greedines we have  almost-greediness, where instead of taking the best approximation under elements of the form $y=\sum_{n\in A}a_n\mathbf x_n$ for any sequence of scalars $(a_n)_{n\in A}$, we have the best approximation by projections (\cite{DKKT}).
	
	\begin{definition}
		We say that a Markushevich basis $\mathcal B$ in a quasi-Banach space $\mathbb X$ is almost-greedy if there is a positive constant $C$ such that for every $x\in\SX$,
		\begin{eqnarray}\label{defalmost}
			\Vert x-P_A(x)\Vert\leq C\inf_{\vert B\vert\leq\vert A\vert}\Vert x-P_B(x)\Vert,
		\end{eqnarray}
		whenever $A$ is a finite greedy set of $x$. The least constant verifying \eqref{defalmost} is denoted by $C_{al}[\mathcal B,\SX]=C_{al}$ and we say that $\mathcal B$ is $C_{al}$-almost-greedy.
	\end{definition}
	As for greediness, the definition was introduced in the context of Schauder bases in Banach spaces in \cite{DKKT}, but recently, in \cite{AABW}, the authors introduced the notion for general Markushevich bases in quasi-Banach spaces. It is clear  that we have the following relations:
	\begin{center}
		greediness $\Rightarrow$ almost-greediness $\Rightarrow$ quasi-greediness,
	\end{center}
	where the first implication is due to the fact that $\sigma_m(x)\leq \Vert x-P_A(x)\Vert$ for any set $A$ of cardinality less than or equal to $m\in\mathbb N$, and the second implication is due to $\inf_{\vert B\vert\leq m}\Vert x-P_B(x)\Vert\leq \Vert x\Vert$ taking $B=\emptyset$. As in the case of greediness, we have a characterization.

	\begin{theorem}[{\cite{DKKT,AABW}}]\label{thalmost}
		Let $\mathcal B$ be a Markushevich basis in a quasi-Banach space. The following are equivalent:
		\begin{enumerate}
			\item $\mathcal B$ is almost-greedy.
			\item $\mathcal B$ is quasi-greedy and democratic.
			\item $\mathcal B$ is quasi-greedy and super-democratic.
		\end{enumerate}
	\end{theorem}
	
	\begin{example}
		The example constructed in \ref{ex1} is also democratic since
		$$\left(\sum_{n\in A} 1\right)^{1/2}=\vert A\vert^{1/2},$$
		and
		$$\sum_{j\in A}\dfrac{1}{\sqrt{j}}\approx \sqrt{\vert A\vert},$$
		so
		$$\Vert \one_A\Vert \approx \sqrt{\vert A\vert},\; \forall \vert A\vert<\infty.$$
	\end{example}
	
	\begin{example}
		To construct now one example of a quasi-greedy and non-democratic basis we also use Example \ref{ex1} but with the following considerations.
		
		We write $\SX\oplus\mathbb Y$ for the Cartesian product of the Banach spaces $\SX$ and $\mathbb Y$ endowed with the norm
		$$\Vert (f,g)\Vert=\max\lbrace\Vert f\Vert,\Vert g\Vert\rbrace,\; f\in\SX, g\in\mathbb Y.$$ 
		Now, given a basis $\mathcal B_1=(\xx_n)_{n\in\mathbb N}$ in $\SX$ and a basis $\mathcal B_2=(\mathbf y_n)_{n\in\mathbb N}$ in the space $\mathbb Y$, the direct sum basis $\mathcal B_1\oplus\mathcal B_2=(\mathbf u_n)_{n\in\mathbb N}$ in  $\SX\oplus\mathbb Y$  is given by
		$$\mathbf u_{2n-1}=(\xx_n,0),\; \mathbf u_{2n}=(0,\mathbf y_n),\; n\in\mathbb N.$$
		Hence, consideraring $\SX$ the space of the Example \ref{ex1} and $\mathbb Y=c_0$, we can take the natural canonical basis $(\mathbf u_n)_{n\in\mathbb N}$ in $\SX\oplus\mathbb Y$. Then, the basis is quasi-greedy since  due to is quasi-greedy in both spaces: consider $f\in\mathbb X\oplus\mathbb Y$, then
		$$x\sim \sum_{n=1}^\infty \mathbf u_n^*(x)\mathbf u_n=\left(\sum_{n=1}^\infty \mathbf x_{2n}^*(x)\mathbf x_{2n},\sum_{n=1}^\infty \mathbf y_{2n-1}^*(x)\mathbf y_{2n-1}\right).$$
		If $G$ is a greedy set of $x\in\SX\oplus\mathbb Y$, decomposing $G=G_{1,\mathbb X}\cup G_{2,\mathbb Y}$ the corresponding greedy sets in the spaces $\mathbb X$ and $\mathbb Y$,
		$$\Vert P_G(x)\Vert_{\SX\oplus\mathbb Y}=\max\lbrace \Vert P_{G_{1,\mathbb X}}(x)\Vert_{\SX},\Vert P_{G_{2,\mathbb Y}}(x)\Vert_{\mathbb Y}\rbrace\lesssim \max\lbrace \Vert x\Vert_{\SX},\Vert x\Vert_{\mathbb Y}\rbrace=\Vert x\Vert_{\SX\oplus\mathbb Y}.$$
		Now, we can take $A=\lbrace 1,3,\dots,2m-1\rbrace$ and $B=\lbrace 2,4,\dots, 2m\rbrace$. Hence,
		$$\Vert \one_A\Vert_{\SX\oplus\mathbb Y}=\left\Vert\sum_{j=1}^{m} \xx_j\right\Vert_{\SX}=\sqrt{m},$$
		and
		$$\Vert \one_B\Vert_{\SX\oplus\mathbb Y}=\left\Vert\sum_{j=1}^{m} \mathbf y_j\right\Vert_{c_0}=1.$$
		Thus,
		$$\dfrac{\Vert\one_A\Vert_{\SX\oplus\mathbb Y}}{\Vert\one_B\Vert_{\SX\oplus\mathbb Y}}=\sqrt{m},$$
		so the basis is not democratic.
	\end{example}
	
	As a novelty, since we have not found a similar result in the literature, we will provide a characterization of almost-greedy bases where we will use elements in the class $\A$. For a quasi-Banach space $\SX$ and a basis $\mathcal B=(\xx_n)_{n\in\mathbb N}$ in $\SX$, we define the set
	$$\A:=\lbrace x\in\SX : \vert \xx_n^*(x)\vert\neq\vert\xx_j^*(x)\vert\; \forall n\neq j, n,j\in\supp(x)\rbrace.$$
	The advantage of working on this set $\A$ is that the greedy sums are unique.
	\begin{theorem}
		Let $\mathcal B$ be a Markushevich basis in a $p$-Banach space $\SX$ with $0<p\leq 1$. The following are equivalent:
		\begin{enumerate}
			\item[1)] $\mathcal B$ is almost-greedy.
			\item[2)] $\mathcal B$ is almost-greedy for elements $x\in\mathbb X_d$.
			\item[3)] $\mathcal B$ is quasi-greedy and democratic.
		\end{enumerate}
	\end{theorem}
	
	\begin{Proof}
		Of course, 1) implies 2). Now, assume that we are in the condition 2), that is, for all $x\in\mathbb X_d$, there is $C>0$ such that
		$$\Vert x-P_G(x)\Vert\leq C\inf_{\vert B\vert\leq\vert G\vert}\Vert x-P_B(x)\Vert.$$
		First, to show 3), we prove quasi-greediness. Take $x\in\mathbb X$ with finite support ande define $m:=\vert \text{supp}(x)\vert$. Take $\varepsilon>0$ and a sequence $(\varepsilon_i)_{i=1}^m$ such that $\varepsilon>\varepsilon_1>\varepsilon_2>...$ and define the element
		$$f':=\sum_{n\in \text{supp}(x)}(\mathbf x_n^*(x)+\text{sign}(\mathbf x_n^*(x))\varepsilon_n^{1/p})\mathbf x_n.$$
		Of course, if $G$ is a greedy set for $x$ it is also greedy for $f'$. Hence,  
		
		\begin{eqnarray}\label{qg1}
			\nonumber\Vert x-P_G(x)\Vert^p &\leq& \Vert x-f'\Vert^p +\Vert f'-P_G(f')\Vert^p + \Vert P_G(f')-P_G(x)\Vert^p\\
			\nonumber&\leq& (1+\Vert P_G\Vert^p)\Vert f'-x\Vert^p+ C^p\Vert f'\Vert^p\\
			&\leq& (2+\Vert P_G\Vert^p)\Vert f'-x\Vert^p+C\Vert x\Vert^p
		\end{eqnarray}
		Now, we can estimate $\Vert f'-x\Vert^p$ as follows:
		
		\begin{eqnarray}\label{qg2}
			\Vert f'-x\Vert^p&=&\left\Vert \sum_{n\in \text{supp}(x)}\text{sign}(\mathbf x_n^*(x))\varepsilon_n^{1/p}\mathbf x_n\right\Vert^p\leq \varepsilon\vert \text{supp}(x)\vert\mathbf c_2^p 
		\end{eqnarray}
		
		Adding up \eqref{qg2} to \eqref{qg1}, we obtain
		$$\Vert x-P_G(x)\Vert^p\leq C\Vert x\Vert^p+\varepsilon\vert \text{supp}(x)\vert\mathbf c_2^p (2+\Vert P_G\Vert^p).$$
		Taking now $\varepsilon\rightarrow 0$, we obtain quasi-greediness for elements with finite support and, applying the result of density \cite[Corollary 7.3]{BBquasi}, the basis is quasi-greedy for every element $x\in\mathbb X$.
		
		We prove now democracy. For that, take $A, B\subset\mathbb N$ two disjoint finite sets such that $\vert A\vert \leq \vert B\vert$. Take $\varepsilon>0$ and define the element
		$$x:=\sum_{n\in A}(1-\varepsilon^n)\mathbf{x}_n+\sum_{n\in B}(1+\varepsilon^n)\mathbf x_n.$$
		Hence,
		\begin{eqnarray}\label{demlabel1}
			\nonumber\Vert \sum_{n\in A}(1-\varepsilon^{n/p})\mathbf{x}_n\Vert&=&\Vert x-\mathcal G_{\vert B\vert}(x)\Vert\leq C\Vert x-\sum_{n\in A}(1-\varepsilon^{n/p})\mathbf{x}_n\Vert\\
			&=&C\Vert \sum_{n\in B}(1+\varepsilon^n)\mathbf{x}_n\Vert
		\end{eqnarray}
		Applying now the $p$-power,
		\begin{eqnarray}\label{demlabel2}
			\Vert \sum_{n\in A}(1-\varepsilon^{n/p})\mathbf{x}_n\Vert^p\geq \Vert \mathbf 1_A\Vert^p-\Vert \sum_{n\in A}\varepsilon^{n/p}\mathbf{x}_n\Vert^p,
		\end{eqnarray}
		and
		\begin{eqnarray}\label{demlabel3}
			\Vert \sum_{n\in B}(1+\varepsilon^{n/p})\mathbf{x}_n\Vert^p\leq \Vert \mathbf 1_B\Vert^p+\Vert \sum_{n\in B}\varepsilon^{n/p}\mathbf{x}_n\Vert^p.
		\end{eqnarray}
		Putting \eqref{demlabel2} and \eqref{demlabel3} in \eqref{demlabel1}, we obtain
		\begin{eqnarray}\label{demlabel4}
			\Vert \mathbf 1_A\Vert^p\leq C^p\Vert \mathbf 1_B\Vert^p + C^p\Vert \sum_{n\in B}\varepsilon^{n/p}\mathbf{x}_n\Vert^p+\Vert \sum_{n\in A}\varepsilon^{n/p}\mathbf{x}_n\Vert^p.
		\end{eqnarray}
		Now, since
		$$\Vert\sum_{n\in B}\varepsilon^{n/p} \mathbf x_n\Vert^p \leq \varepsilon \mathbf{c}_2^p\vert B\vert,$$
		and
		$$\Vert\sum_{n\in A}\varepsilon^{n/p} \mathbf x_n\Vert^p \leq \varepsilon \mathbf{c}_2^p\vert A\vert.$$
		Applying these last inequalities in \eqref{demlabel4},
		$$\Vert\mathbf 1_A\Vert^p \leq C^p\Vert \mathbf 1_B\Vert^p+\varepsilon\mathbf c_2^p\vert B\vert(C^p+1).$$
		Taking now $\varepsilon\rightarrow 0$, we obtain the the basis is democratic for disjoint sets with constant $C$. Taking now $A, B$ two general sets with $\vert A\vert\leq \vert B\vert$ and other set $D> A\cup B$ such that $\vert A\vert=\vert D\vert$,
		$$\dfrac{\Vert \one_A\Vert}{\Vert \one_B\Vert}=\dfrac{\Vert \one_A\Vert}{\Vert \one_C\Vert}\dfrac{\Vert \one_C\Vert}{\Vert \one_B\Vert}\leq C^2,$$
		so the basis is democratic and 3) is proved. The implication 3) $\Rightarrow$ 1) is the main characterization of almost-greediness (Theorem \ref{thalmost}).
	\end{Proof}

	In \cite{DKK2003}, the authors study, in some sense, the ``distance" between greedy and almost-greedy bases as follows.
	
	\begin{theorem}[\cite{AABW,DKK2003}]
		Let $\mathcal{B}$ be a Markushevich basis in a quasi-Banach space $\mathbb{X}$. Then the basis is almost-greedy if and only if for every (or for some) $\lambda > 1$, there exists a constant $C_\lambda$ such that for every $m \in \mathbb{N}$ and every greedy set $A$ of $x$ with cardinality $\lceil \lambda m \rceil$,
		$$\Vert x-P_A(x)\Vert \leq C_\lambda \sigma_m(x).$$
		Moreover, if $\mathbb{X}$ is a $p$-Banach space, the optimal constant $C_\lambda$ in the previous inequality satisfies
		$$C_\lambda \lesssim  \lceil(\lambda-1)^{-1}\rceil^{1/p}.$$ 
	\end{theorem}
	
	As we can see from the theorem, the fact that a basis is almost-greedy allows us to recover the best approximation error $\sigma_m(x)$, but by taking larger greedy sums. However, naturally, we would like to take larger greedy sums but with $\lambda \rightarrow 1$ but, in this case, the constant $C_\lambda$ explodes. Therefore, in order to see if it is possible to somehow recover the best approximation error through almost-greedy bases, S. J. Dilworth, N. J. Kalton, and D. Kutzarova propose the following modification of the TGA: let $x \in \mathbb{X}$, and consider a finite greedy set $A$ of $x$. We define a Chebyshev sum of order $n := \vert A \vert$ of $x$ as any element $\mathcal{CG}_n[\mathcal{B},\mathbb{X}](x) = \mathcal{CG}_n(x)$ of the form $\sum_{j \in A} a_j \xx_j \in \text{span}\{\xx_i : i \in A\}$ such that
	$$\Vert x- \mathcal{CG}_n(x) \Vert = \min \left\lbrace \left\Vert x- \sum_{i \in A} a_i \xx_i \right\Vert : a_i \in \mathbb{F}, \forall i \in A \right\rbrace.$$
	
	\begin{definition}
		Let $\mathcal{B}$ be a Markushevich basis in a quasi-Banach space $\mathbb X$. We say that the basis is semi-greedy if for every $x \in \mathbb{X}$, there exists a constant $C$ such that
		\begin{eqnarray*}\label{sec0-semi}
			\min \left\lbrace \left\Vert x- \sum_{n \in A} a_n \xx_n \right\Vert : a_n \in \mathbb{F}, \forall n \in A \right\rbrace \leq C \sigma_{\vert A \vert}(x),
		\end{eqnarray*}
		whenever $A$ is any finite greedy set of $x$. The smallest constant verifying inequality \eqref{sec0-semi} is denoted by $C_{sg}[\mathcal B,\mathbb X]=C_{sg}$, and we say that the basis $\mathcal{B}$ is $C_{sg}$-semi-greedy.
	\end{definition}
	
	In \cite{DKK2003}, the authors proved the following result.
	
	\begin{theorem}
		Let $\mathcal B$ be a Schauder basis in a Banach space $\mathbb X$ with finity cotype. Then, the basis is almost-greedy if and only if the basis is semi-greedy.
	\end{theorem}
	
	In other words, we can recover the error $\sigma_m(x)$ from almost-greedy bases by using Chebyshev-type greedy sums. The only ``problem'' with this result is that we have the condition of working with spaces of finite cotype, and this condition does not appear in any characterization of other types of greedy-like bases. For this reason, more authors have become interested in these bases. For instance, in \cite{B2019}, the author improved the result of \cite{DKK2003} by eliminating the finite cotype condition but the author also worked under the condition of Schauder bases, but this type of bases were relaxed some years later in the following result.
	
	\begin{theorem}[{\cite{BL}}]
		Let $\mathcal B$ be a Markushevich basis in a Banach space $\mathbb X$. Then, the basis is almost-greedy if and only if the basis is semi-greedy.
	\end{theorem}

	It is clear from this last result that, in the context of Banach spaces, having an almost-greedy basis makes it possible to improve the efficiency of the TGA using Chebyshev sums. The question that now arises is: what about now the characterization in the context of quasi-Banach or $p$-Banach spaces? In \cite{BCH}, the authors showed the equivalence between almost-greedy and semi-greedy bases but under the condition of Schauder bases.
	
	\begin{theorem}
		Let $\mathcal B$ be a Schauder basis in a quasi-Banach space. The basis is semi-greedy if and only if the basis is almost-greedy.
	\end{theorem}
	
	\textbf{Question 3.} Is it possible to remove the condition of Schauder in the characterization of semi-greediness in quasi-Banach spaces?
	
	We want to remark that the proof provided in \cite{BL} for the characterization of almost-greediness with semi-greediness for general Markushevich bases requires one argument that is only valid in the Banach setting using the relation between the space $\mathbb X$ and the bidual one $\mathbb X^{**}$.
	
	\subsection{Isometric case of greedy-like bases}
	
	Related to semi-greediness, we have another open question about the isometric case. To explain that, we give a little review about the constant 1 for quasi-greedy, greedy and almost-greedy bases.
	
	In the first characterization of greedy bases in terms of unconditionality and democracy in Banach spaces given in \cite{KT}, the authors gave the following estimates:
	$$\max\lbrace \Delta_d, K\rbrace\leq C_g\leq K(1+\Delta_d).$$
	Hence, if $\mathcal B$ is 1-democratic and 1-unconditional, the basis is $C_g$-greedy with $C_g\leq 2$ and, in 2006, in \cite{AW}, the authors showed that the constant 2 is optimal. For that reason, the question here is natural: under what conditions is it possible to recover $C_g=1$?  To answer this question, F. Albiac and P. Wojtasczyk introduced the Property (A): given $x=\sum_{n\in S}\xx_n^*(x)\xx_n$ with $\vert S\vert<\infty$, we write $M(x):=\lbrace n\in S : \vert \xx_n^*(x)\vert=\max_{j\in S}\vert\xx_j^*(x)\vert\rbrace$. Hence, the basis has the Property (A) whenever
	$$\Vert x\Vert=\left\Vert \sum_{n\in M(x)}\varepsilon_n\xx_n^*(x)\xx_{\lambda(n)}+P_{M^c(x)}(x)\right\Vert,$$
	for all injective maps $\lambda: S\rightarrow\mathbb N$ such that $\lambda(j)=j$ if $j\not\in M(x)$ and $\varepsilon_n\in\lbrace\pm 1\rbrace$ with $\varepsilon_n=1$ whenever $\lambda(n)=n$ for $n\in M(x)$.
	\begin{theorem}
		A basis in a (real) Banach space is $1$-greedy if and only if the basis is $1$-unconditional and has the Property (A).
	\end{theorem}
	Some years later, in \cite{DKOS}, the authors rewrite this property in the following sense: a basis has the Property (A) if
	$$\Vert x+\one_{\boldsymbol{\varepsilon} A}\Vert = \Vert x+\one_{\boldsymbol{\eta} B}\Vert,$$
	whenever $x\in\SX$ with $\vert\xx_n^*(x)\vert\leq 1$ for all $n\in \supp(x)$, $\vert A\vert\leq\vert B\vert<\infty$ with $A\cap B=\emptyset$ and $\supp(x)\cap (A\cup B)=\emptyset$ and $\boldsymbol{\varepsilon}\in\mathcal E_A$, $\boldsymbol{\eta}\in\mathcal E_B$.
	
	In fact, this condition was renamed as symmetry for largest coefficients since is a weaker condition than the usual definition of symmetry. We remind that a basis $(\mathbf x_n)_{n\in\mathbb N}$ is symmetric if it is equivalent to $(\mathbf x_{p(n)})_{n\in\mathbb N}$ for every permutation $p$ on $\mathbb N$, that is, there is $\mathcal S>0$ such that 
	$$\left\Vert \sum_n a_n\mathbf x_{p(n)}\right\Vert\leq\mathcal S\left \Vert \sum_n a_n\mathbf x_{n}\right\Vert,$$
	for any permutation $p$ and any sequence $(a_n)_{n\in\mathbb N}\in c_{00}$. Then, it is obvious that $1$-symmetry implies $1$-unconditional and $1$-democratic but, in fact, in \cite[Theorem 2.5]{AW} it is proved that if the basis is $1$-symmetric, then the basis is $1$-greedy (the converse is false as we can see in \cite{AW}). Then, Property (A) was renamed to ``symmetry for largest coefficients” since we take the only over the largest coefficients (in modulus).

	Related to $1$-almost-greediness, in \cite{AA2}, the authors proved the following characterization.
	
	\begin{theorem}
		A basis is $1$-almost-greedy if and only if the basis has the Property (A).
	\end{theorem}
	This theorem shows that the Property (A) is an stronger condition than quasi-greediness. Hence, other open question in the field is whether Property (A) implies unconditionality, meaning whether Property (A) is strong enough to imply that the basis is greedy.
	
	Here, we discuss the isometric case about semi-greediness. In  \cite{BGG}, the authors present the following result and also analyze the improvement of the estimate obtained compared to the one proven in article \cite{BCH}.
	\begin{theorem}
		Let $\mathcal B$ be a bi-monotone Schauder basis in a $p$-Banach space. If $\mathcal B$ is $1$-semi-greedy, then the basis is $3^{2/p}$-almost-greedy.
	\end{theorem}
	
	\begin{Remark}
		A basis is said to be bi-monotone in a quasi-Banach space $\mathbb X$ when
		$$\max\lbrace \Vert S_m(x)\Vert, \Vert x-S_m(x)\Vert\rbrace\leq \Vert x\Vert,\; \forall m\in\mathbb N, \forall x\in\SX.$$
	\end{Remark}
	
	Hence, here, the natural question is the following one.
	
	\textbf{Question 4.} Is it true that $1$-almost-greediness is equivalent to $1$-semi-greediness in the context of Banach (and quasi-Banach) spaces? 
	
	Another line of study in the isometric case of greedy-type bases is the existence of renormings, that is, if my basis is $C$-greedy (or another type of greedy-like basis), does there exist an equivalent norm such that the basis becomes 1-greedy? In general, a renorming $\| \cdot \|_0$ of $(\mathbb{X}, \| \cdot \|)$ has the form
	\[
	\|x\|_0 = \max \{a \|x\|, \|T(x)\|_\mathbb{Y}\}
	\]
	for some $0 < a < \infty$ and some bounded linear operator $T$ from $\mathbb{X}$ into a Banach space $\mathbb{Y}$.
	
	The problem with the renormings in Banach spaces is that, as we can see in the last equality, $T$ has to be linear and the greedy sums are not linear, so for that reason there is not a theory about renormings for greedy-like bases in Banach spaces. One paper where this idea is posed is \cite{DKOS}, where the authors proved that, for a fixed $\epsilon > 0$, it is possible to find a renorming in $L_p$, $1 < p < \infty$, such that the Haar system is $(1 + \epsilon)$-greedy, but it is open whether one can get the constant $1$.
	
	The situation change when we work with strictly quasi-Banach spaces as we can see in \cite{AABW}, where the authors proved the following lemma.
	
	\begin{lemma}[{\cite[Lemma 12.1]{AABW}}]
		Let $(\mathbb{X}, \| \cdot \|)$ be a quasi-Banach space. Assume that $\| \cdot \|_0: \mathbb{X} \to [0, +\infty)$ is such that, for every $t \in \mathbb{F}$ and for every $x \in \mathbb{X}$,
		\begin{itemize}
			\item $\|tx\|_0 = |t|\|x\|_0$,
			\item $\|x\|_0 \approx \|x\|$.
		\end{itemize}
		Then $\| \cdot \|_0$ is a renorming of $\| \cdot \|$.
	\end{lemma}
	
	This lemma allows us to have renormings of (strictly) quasi-Banach spaces based on non-linear operators. In fact, in \cite{AABW}, the authors proved that if $\mathcal B$ is $C$-quasi-greedy (resp. $C$-almost-greedy or $C$-greedy), there is a renorming such that the basis is $1$-quasi-greedy (resp. $1$-almost-greedy or $1$-greedy). In \cite{B}, the author did the same for partially-greedy bases. Thus, the open question here is the following one.
	
	\textbf{Question 5.} If $\mathcal B$ is $C$-semi-greedy in a quasi-Banach space, is there a renorming such that the basis is $1$-semi-greedy?\bigskip
	
	\textbf{Acknowledgments}: the author thanks the referees for their comments and also expresses gratitude to the Professor F. J. Fernández for his feedback.
	
	\textbf{Declarations}\bigskip
	
	\textbf{Ethics approval and consent to participate}\bigskip
	Not applicable.
	
	\textbf{Availability of data and material}\bigskip
	Not applicable.

\end{document}